\documentclass[a4paper,12pt]{article}
\usepackage{geometry}
\usepackage{amssymb}
\usepackage{xypic}
\usepackage[all,2cell]{xy}
\usepackage{graphics}
\usepackage{amsmath}
\usepackage{fullpage}
\usepackage{theorem}
\usepackage{amstext}
\usepackage{mathrsfs}
\usepackage{calligra}
\usepackage{xcolor, soul}
\sethlcolor{yellow}

\newtheorem{theorem}{Theorem}[section]

\newtheorem{proposition}{Proposition}[section]

\newtheorem{lemma}{Lemma}[section]
\theorembodyfont{\upshape}

\newtheorem{definition}{Definition}[section]

\newenvironment{proof}[1][Proof]{\textbf{#1.} }{\ \rule{0.5em}{0.5em}}
\begin{document}
	
		\title{Fuzzy semigroups via semigroups}
	\author{
		Anjeza Krakulli \\
		Universiteti Aleksand\"er Moisiu \\
		Fakulteti i Teknologjis\"e dhe Informacionit \\
		Departamenti i Matematik\"es, Durr\"es \\
		anjeza.krakulli@gmail.com \\
		Elton Pasku \\
		Universiteti i Tiran\"es \\
		Fakulteti i Shkencave Natyrore \\
		Departamenti i Matematik\"es, Tiran\"e \\
		elton.pasku@fshn.edu.al \\
	}
	
	\date{}
	
	\maketitle

\begin{abstract}
The theory of fuzzy semigroups is a branch of mathematics that arose in early 90's as an effort to characterize properties of semigroups by the properties of their fuzzy subsystems which include, fuzzy subsemigroups and their alike, fuzzy one (resp. two) sided ideals, fuzzy quasi-ideals, fuzzy bi-ideals etc. To be more precise, a fuzzy subsemigroup of a given semigroup $(S,\cdot)$ is just a $\wedge$-prehomomorphism $f$ of $(S,\cdot)$ to $([0,1],\wedge)$. Variations of this, which correspond to the other before mentioned fuzzy subsystems, can be obtained by imposing certain properties to $f$. It turns out from the work of Kuroki, Mordeson, Malik and that of many of their descendants, that fuzzy subsystems play a similar role to the structure theory of semigroups that play their non fuzzy analogues. The aim of the present paper is to show that this similarity is not coincidental. As a first step to this, we prove that there is a 1-1 correspondence between fuzzy subsemigroups of $S$ and subsemigroups of a certain type of $S\times I$. Restricted to fuzzy one sided ideals, this correspondence identifies the above fuzzy subsystems to their analogues of $S\times I$. Using these identifications, we prove that the characterization of the regularity of semigroups in terms of fuzzy one sided ideals and fuzzy quasi-ideals can be obtained as an implication of the corresponding non fuzzy analogue. In a further step, we give new characterizations of semilattices of left simple semigroups in terms of left simple fuzzy subsemigroups, and of completely regualr semigroups in terms of completely simple fuzzy subsemigroups. Both, left simple fuzzy subsemigroups, and completely simple fuzzy subsemigroups are defined here for the first time, and the corresponding characterizations generalize well known characterizations of the corresponding semigroups. 
\end{abstract}

\section{Introduction}

Given a semigroup $(S,\cdot)$, a fuzzy subsemigroup of $S$ is a function $f:S \rightarrow [0,1]$ such that for all $a,b \in S$, $f(ab) \geq f(a) \wedge f(b)$. Thus, a fuzzy subsemigroup is just a prehomomorphism $f:(S, \cdot) \rightarrow ([0,1], \wedge)$ of semigroups. A fuzzy subsemigroup as above will be called a fuzzy left ideal (resp. right ideal) of $S$ if for all $a,b \in S$, $f(ab) \geq  f(b)$ (resp. $f(ab) \geq f(a) $). The composite of two fuzzy subsets $f,g: S \rightarrow [0,1]$ is given by
$$(f \circ g)(a)=\left\{ \begin{array}{ccc}
	{\vee}_{a=bc}(f(b) \wedge g(c)) & if & \exists x,y \in S, a=xy \\
	0 & if & \forall x,y \in S, a \neq xy.
\end{array}\right.$$
The operation $\circ$ is associative. The definition of $\circ$ is needed in defining fuzzy quasi-ideals which are maps $q:S \rightarrow [0,1]$ such that $q \circ S \cap S \circ q \subseteq q$. In more details, the inclusion here means that for every $a \in S$, $(q \circ S)(a) \wedge (S \circ q)(a) \leq q(a)$, and $S$ is the fuzzy subsemigroup which maps every $a \in S$ to 1. More generally, for every subsemigroup $A$ of $S$, the characteristic map $\chi_{A}$, defined by
$$\chi_{A}(a)=\left\{ \begin{array}{ccc}
1 & if & a \in A\\
0 & if & a \notin A
\end{array}\right.$$
is proved to be a fuzzy subsemigroup of $S$. The above notions and other ones that are needed to understand the rest of the paper can be found mostly in the monograph \cite{FS}. Necessary concepts from semigroup theory are found in \cite{JMH} and \cite{OS}. From \cite{OS}, we are presenting below theorem 9.3 whose fuzzy counterpart is proved partially in \cite{FS}. We prove here its fuzzy counterpart theorem \ref{osreg}, including the part already proved in \cite{FS}, not in a direct fashion as in \cite{FS}, but as an implication of theorem 9.3 after having established a 1-1 correspondence between fuzzy one sided ideals of a semigroup $S$ on the one hand, and one sided ideals of $S \times [0,1]$ on the other hand. Theorem 9.3 of \cite{OS} reads as follows.
\begin{theorem}
	The following are equivalent.
	\begin{itemize}
		\item [(i)] $S$ is a regular semigroups.
		\item[(ii)] For every right ideal $R$ and every fuzzy left ideal $L$, we have $R \cap L=R L$.
		\item[(iii)] For every right ideal $R$ and every left ideal $L$ of $S$, (a) $R^{2}=R$, (b) $L^{2}=L$, (c) $RL$ is a quasi-ideal of $S$. 
		\item[(iv)] The set of all quasi-ideals of $S$ forms a regular semigroup.
		\item[(v)] For every quasi ideal $Q$ of $S$, $Q S  Q=Q$.
	\end{itemize}
\end{theorem}
The last section of this paper is devoted to the semilattice decomposition of left regular semigroups whose left ideals are also right ideals. It is proved in \cite{Saito} that such semigroups have a nice decomposition into smaller and easier to understand subsemigroups. More specifically we have from \cite{Saito} the following theorem.
\begin{theorem} \label{tsaito}
	For a semigroup $(S, \cdot)$ the following conditions
	are equivalent:
	\begin{itemize}
		\item[(1)] $S$ is a semilattice of left simple semigroups.
		\item[(2)] $L_{1} \cap L_{2}=L_{1}L_{2}$ for every left ideals $L_{1}$ and $L_{2}$ of $S$.
		\item[(3)]  The set of all left ideals of $S$ is a semilattice under the multiplication of
		subsets.
		\item[(4)] $S$ is left regular and every left ideal of $S$ is two-sided.
	\end{itemize}
\end{theorem}
There is also a version of this theorem in \cite{FS} (Theorem 4.1.3) where condition (2), (3) and (4) are fuzzyfied by replacing the term left ideal (resp. ideal) by fuzzy left ideal (resp. fuzzy ideal). We prove in our theorem \ref{tsaito+fuzzy} that condition (1) can also be fuzzyfied. For this we have defined what a fuzzy semilattice of left simple fuzzy subsemigroups is. In the same spirit with the above we give a fuzzy version of Theorem 4.1.3 of \cite{JMH} by characterizing completely regular semigroups in terms of what are defined here to be fuzzy semilattices of completely simple fuzzy subsemigroups.

\section{A relationship between fuzzy subsemigroups and semigroups}

Assume we are given a fuzzy subsemigroup $f:S \rightarrow [0,1]$ of a semigroup $(S,\cdot)$. Besides the semigroup $(S,\cdot)$ we consider the semilattice $([0,1],\wedge)$. We sometimes write $I$ instead of $[0,1]$ and let $I^{\ast}=I \setminus  \{0\}$. We can obviously regard $f$ as a prehomomorphism from $(S,\cdot)$ to $([0,1],\wedge)$. With this data at hand, we define a subsemigroup of the direct product semigroup $(S, \cdot) \times ([0,1],\wedge)$ by letting
$$\mathfrak{S}(S,f)=\{(b, y) \in S \times [0,1]: f(b)\geq y\}.$$
This is indeed a subsemigroup of $S \times [0,1]$ since for every $(a,x), (b,y) \in \mathfrak{S}(S,f)$ we have that $f(ab)\geq f(a) \wedge f(b) \geq x \wedge y$, and then,
$$(a,x)(b,y)=(ab, x \wedge y) \in \mathfrak{S}(S,f).$$
The semigroup $\mathfrak{S}(S,f)$ satisfies the following three conditions
\begin{itemize}
	\item [(i)] $\pi_{1}(\mathfrak{S}(S,f))=S$ where $\pi_{1}$ is the projection in the first coordinate,
	\item[(ii)] for every $b \in S$, $(b,\text{sup}\{y \in [0,1]: (b,y) \in \mathfrak{S}(S,f)\}) \in \mathfrak{S}(S,f)$,
	since $\text{sup}\{y \in [0,1]: (b,y) \in \mathfrak{S}(S,f)\}=f(b)$, and
	\item[(iii)] $(b,y) \in \mathfrak{S}(S,f)$ for every $y \leq f(b)$.
\end{itemize}
Conversely, every subsemigroup $\Sigma$ of the direct product $S \times [0,1]$ which satisfies the following three conditions
\begin{itemize}
	\item [(s-i)] $\pi_{1}(\Sigma)=S$ where $\pi_{1}$ is the projection in the first coordinate,
	\item[(s-ii)] for every $b \in S$, $(b,\text{sup}\{y \in [0,1]: (b,y) \in \Sigma\}) \in \Sigma$, and
	\item[(s-iii)] $(b,y) \in \Sigma$ for every $y \leq \text{sup}\{z \in [0,1]: (b,z) \in \Sigma\}$,
\end{itemize}
gives rise to a fuzzy subsemigroup $f: S \rightarrow [0,1]$ such that $\mathfrak{S}(S,f)=\Sigma$. For this we define
$$f: S \rightarrow [0,1] \text{ by } b \mapsto \text{sup}\{y \in [0,1]: (b,y) \in \Sigma\}.$$
From condition (ii) we know that $(b,f(b)) \in \Sigma$. The function $f$ is a fuzzy subsemigroup for if $a,b \in S$, then since $(a,f(a)), (b,f(b)) \in \Sigma$, it follows that 
$$(ab, f(a) \wedge f(b))=(a,f(a))(b,f(b)) \in \Sigma,$$
and then $f(ab) \geq f(a) \wedge f(b)$. Finally, 
\begin{align*}
	(b, y) \in \mathfrak{S}(S,f) & \Leftrightarrow f(b) \geq y\\
	& \Leftrightarrow (b,y) \in \Sigma && \text{(from condition (s-iii))}
\end{align*}
The set of subsemigroups $\Sigma$ of $S \times I$ as above is denoted here by $F(S \times I)$. Also we write $\mathfrak{F}(S)$ for the set of all fuzzy subsemigroups of $S$. Define now
$$\Psi: \mathfrak{F}(S) \rightarrow F(S \times I) \text{ by } f \mapsto \mathfrak{S}(S,f).$$
\begin{proposition} \label{fs=s}
The map $\Psi$ is bijective.
\end{proposition}
\begin{proof}
That $\Psi$ is correct and onto was proved above. To prove that $\Psi$ is injective, we let $f,g \in \mathfrak{F}(S)$ such that $\mathfrak{S}(S,f)=\mathfrak{S}(S,g)$. Then, for every $a \in S$, $(a,g(a)) \in \mathfrak{S}(S,g)$, hence $(a,g(a)) \in \mathfrak{S}(S,f)$, and then $g(a) \leq f(a)$. In a symmetric way one can show that $f(a) \leq g(a)$, therefore $g(a)=f(a)$.
\end{proof}\\

We conclude this section with a further remark on the relationships between subsemigroups of $S \times I$ and fuzzy subsemigroups of $S$. Once we associated every fuzzy subsemigroup of $S$ with a subsemigroup of $S \times I$, we would like to establish a reverse correspondence. It is obvious that not every subsemigroup of $S \times I$ is in the image of $\Psi$. For instance, if $t \in I^{\ast}$ is fixed, then every subsemigroup $\Sigma$ of $S \times I$ with $\pi_{2}(\Sigma)=t$ cannot be in the image of $\Psi$. But still we can define a fuzzy subsemigroup $\sigma$ in terms of $\Sigma$ and find a relationship between $\mathfrak{S}(S,\sigma)$ and $\Sigma$. 
For this we define $\sigma: S \rightarrow I$ such that
$$\sigma(a)=\left\{ \begin{array}{ccc}
	\alpha & if & a \in \pi_{1}(\Sigma) \\
	0 & if & a \notin \pi_{1}(\Sigma),
\end{array}\right.$$
where $\alpha=\text{sup}\{t \in I: (a,t) \in \Sigma\}$. We prove that $\sigma$ is a fuzzy subsemigroup of $S$. Indeed, let $a, b \in S$. If $ab \notin \pi_{1}(\Sigma)$, then it is obvious that at least one of $a$ or $b$ cannot be in $\pi_{1}(\Sigma)$, consequently $\sigma(ab) =0=\sigma(a) \wedge \sigma(b)$. If $ab \in \pi_{1}(\Sigma)$ we let $\gamma=\text{sup}C$ where $C=\text{sup}\{t \in I: (ab,t) \in \Sigma\}$. If now $a,b \in \pi_{1}(\Sigma)$, we write $A=\{t' \in I: (a,t') \in \Sigma\}$, $B=\{t'' \in I: (b,t'') \in \Sigma\}$, and let $\alpha=\text{sup}A$ and $\beta=\text{sup}B$. We prove that $\alpha \wedge \beta \leq \gamma$ which in turn is equivalent to $\sigma(ab) \geq \sigma(a) \wedge \sigma(b)$.
Observe first that $t' \wedge t'' \in C$ for every $t' \in A$ and $t'' \in B$ since $(ab,t' \wedge t'')=(a,t')(b,t'') \in \Sigma$. Let us assume that $\alpha < \beta$ or equivalently that $\alpha \wedge \beta=\alpha$. In this case, there is $t_{\ast}'' \in B$ such that $t_{\ast}'' > \alpha$, and then
\begin{align*}
\alpha&=\text{sup} A\\
&=\text{sup}\{t' \wedge t_{\ast}'': t' \in A\}\\ 
&\leq \text{sup}\{t' \wedge t'': t' \in A, t'' \in B\} \\
&\leq \text{sup}C && \text{(from our observation)}\\
&=\gamma.
\end{align*}
A similar proof is available when $\alpha > \beta$. Let us now prove the claim when $\alpha=\beta$. There are increasing sequences $\{t'_{n}\} \subseteq A$ and $\{t''_{n}\} \subseteq B$ such that $t'_{n} \rightarrow \alpha$ and $t''_{n} \rightarrow \alpha$. It follows that $t'_{n} \wedge t''_{n} \rightarrow \alpha$, and then
\begin{align*}
\alpha&=\text{sup} A\\
&=\text{sup}\{t'_{n} : n \in \mathbb{N}\}\\ 
&=\text{sup}\{t'_{n} \wedge t''_{n} : n \in \mathbb{N}\}\\
& \leq \text{sup} C && \text{(from the obsevation)}\\
&=\gamma.
\end{align*}
In the three remaining cases when at least one of $a$ or $b$ is not in $\pi_{1}(\Sigma)$ it is straightforward that $\alpha \wedge \beta \leq \gamma$. Now we prove that $\Sigma \subseteq \mathfrak{S}(S, \sigma)$. Indeed, let $(a,t) \in \Sigma$, then from the definition of $\sigma$, $\sigma(a) \geq t$, hence $(a,t) \in \mathfrak{S}(S, \sigma)$. Letting $\text{Sub}(S \times I)$ be the set of subsemigroups of the product semigroup $S \times I$, we define
$$\tilde{\Psi}: \text{Sub}(S \times I) \rightarrow \mathfrak{F}(S) \text{ by } \Sigma \mapsto \sigma,$$
with $\sigma$ described as above. This map is obviously surjecive, and in general non injective since if $\sigma$ is such that $\Sigma \neq \mathfrak{S}(S, \sigma)$, then $\tilde{\Psi}(\Sigma)=\sigma= \tilde{\Psi}(\mathfrak{S}(S, \sigma))$. The right hand side equality above shows that $\tilde{\Psi}$ is a left inverse of $\Psi$.

\section{Similar relationships for other fuzzy subsystems}

Given a fuzzy left ideal $f:S \rightarrow [0,1]$ of a semigroup $(S,\cdot)$, we define a left ideal in the direct product semigroup $S \times [0,1]$ by
$$\mathcal{L}(S,f)=\{(b, y) \in S \times [0,1]: f(b)\geq y\}.$$
This is indeed a left ideal of $S \times [0,1]$ for if $(a,x) \in S \times [0,1]$ and $(b,y) \in \mathcal{L}(S,f)$, then
$$(a,x)(b,y)=(ab, x \wedge y) \in \mathcal{L}(S,f)$$
since 
$$f(ab) \geq f(b) \geq x \wedge f(b) \geq x \wedge y.$$
We also note that $\mathcal{L}(S,f)$ satisfies the following three conditions:
\begin{itemize}
	\item [(i)] $\pi_{1}(\mathcal{L}(S,f))=S$ where $\pi_{1}$ is the projection in the first coordinate,
	\item[(ii)] for every $b \in S$, $(b,\text{sup}\{y \in [0,1]: (b,y) \in \mathcal{L}(S,f)\}) \in \mathcal{L}(S,f)$,
	since $\text{sup}\{y \in [0,1]: (b,y) \in \mathcal{L}(S,f)\}=f(b)$, and
	\item[(iii)] $(b,y) \in \mathcal{L}(S,f)$ for every $y \leq f(b)$.
\end{itemize}

Conversely, every left ideal $L$ of the direct product semigroup $S \times [0,1]$ which satisfies the following three conditions
\begin{itemize}
	\item [(l-i)] $\pi_{1}(L)=S$ where $\pi_{1}$ is the projection in the first coordinate,
	\item[(l-ii)] for every $b \in S$, $(b,\text{sup}\{y \in [0,1]: (b,y) \in L\}) \in L$, and
	\item[(l-iii)] $(b,y) \in L$ for every $y \leq \text{sup}\{z \in [0,1]: (b,z) \in L\}$,
\end{itemize}
gives rise to a fuzzy left ideal of $S$. For this we define
$$f:S \rightarrow [0,1]: b \mapsto \text{sup}\{y \in [0,1]: (b,y) \in L\}.$$
In particular we have that $(b,f(b)) \in L$. Now we show that for all $a,b \in S$, $f(ab) \geq f(b)$. Since $L$ is a left ideal, then for $(a,1) \in S \times [0,1]$ and $(b,f(b)) \in L$, we have 
$$(a,1)(b,f(b))=(ab,f(b)) \in L,$$
hence from the definition of $f$, we have $f(ab)\geq f(b)$. Also we note that the left ideal $\mathcal{L}(S,f)$ arising from the fuzzy left ideal $f$ defined in terms of $L$ is exactly the left ideal $L$ we started with, that is $\mathcal{L}(S,f)=L$. Indeed, 
\begin{align*}
(b,y) \in \mathcal{L}(S,f) & \Leftrightarrow f(b) \geq y\\
& \Leftrightarrow  (b,y) \in L && \text{(from condition (l-iii))}
\end{align*}
A similar result with proposition \ref{fs=s} holds true for fuzzy left ideals. We let $L(S \times I)$ be the set of left ideals of $S \times I$ which satisfy properties (l-i), (l-ii) and (l-iii) above. Also we write $\mathfrak{L}(S)$ for the set of all fuzzy left ideals of $S$. 
\begin{proposition} \label{fl=l}
The restriction $\Psi|_{\mathfrak{L}(S)}$ of $\Psi$ on $\mathfrak{L}(S)$ is a bijection between $\mathfrak{L}(S)$ and $L(S \times I)$.
\end{proposition}
\begin{proof}
It is straightforward that for every fuzzy left ideal $f$, $\mathcal{L}(S,f)=\Psi(f)$, hence $\Psi|_{\mathfrak{L}(S)}: \mathfrak{L}(S) \rightarrow L(S \times I)$ is well defined. On the other hand, every left ideal $L$ of $L(S \times I)$ is in the image of $\Psi$, hence the restriction map is onto. Finally, $\Psi|_{\mathfrak{L}(S)}$ is 1-1 as a restriction of a 1-1 map.
\end{proof}\\
\\
We note by passing that similar results hold true for fuzzy right ideals too. For instance, one can prove that if we let $\mathfrak{R}(S)$ be the set of fuzzy right ideals of $S$, and $R(S \times I)$ the set of right ideals of $S \times I$ which satisfy certain condition similar to those $(l-i,ii,iii)$, then we have,
\begin{proposition} \label{fr=r}
	The restriction $\Psi|_{\mathfrak{R}(S)}$ of $\Psi$ on $\mathfrak{R}(S)$ is a bijection between $\mathfrak{R}(S)$ and $R(S \times I)$.
\end{proposition}

We conclude this section by proving that analogous results to the above hold true for fuzzy quasi-ideals. 
\begin{lemma} \label{lm}
	Whenever $q$ is a fuzzy quasi-ideal and $a \in S$ has a family of factorizations $a=b_{i}c_{i}$ with $i \in J$, then $\forall i \in J$, $q(a) \geq q(b_{i})$ or $q(a) \geq q(c_{i})$.
\end{lemma}
\begin{proof}
	From the definition of a fuzzy quasi-ideal $q$ we have
	\begin{align*}
		q(a) & \geq (\underset{i \in J}{\vee }q(b_{i} )) \wedge (\underset{j \in J}{\vee }  q(c_{j}))\\
		&=\underset{i \in J}{\vee }q(b_{i}) && \text{(if for instance,)}
	\end{align*}
	then $q(a) \geq q(b_{i})$  for all $i \in J$. Similarly, one can show that when 
	$$\text{min}\left(\underset{i \in J}{\vee }q(b_{i}), \underset{j \in J}{\vee }q(c_{j})\right)=\underset{j \in J}{\vee }q(c_{j}),$$
	then $q(a) \geq q(c_{j})$.
\end{proof}\\
\\
For every fixed fuzzy quasi-ideal $q$ of $S$, we let
$$\mathfrak{Q}(S,q)=\{(a,x) \in S \times I: q(a) \geq x\},$$
and prove that $\mathfrak{Q}(S,q)$ is a quasi-ideal of $S \times I$. Indeed, if $(a,x),(b,y) \in \mathfrak{Q}(S,q)$ and $(s,t),(s',t') \in S \times I$ are such that
$$(a,x)(s,t)=(s',t')(b,y) \in \mathfrak{Q}(S,q) \cdot (S \times I) \cap (S \times I) \cdot \mathfrak{Q}(S,q),$$ then
\begin{align*}
	(as,x \wedge t)=(s'b,t' \wedge y), 
\end{align*}
hence $as=s'b$ and $q(as)=q(s'b)$. From Lemma \ref{lm} we have that either $q(as) \geq q(a)$, or $q(s'b) \geq q(b)$. In the first case we have that
$$q(as) \geq q(a) \geq  x \geq x \wedge t,$$
which proves that $(a,x)(s,t) = (as,x \wedge t) \in \mathfrak{Q}(S,q)$. If the second case above occurs, then 
$$q(s'b) \geq q(b) \geq y \geq y \wedge t',$$
consequently, $(s',t')(b,y)=(s'b, t' \wedge b) \in \mathfrak{Q}(S,q)$. We have thus proved that $\mathfrak{Q}(S,q)$ is a quasi-ideal of $S \times I$. The quasi-ideal $\mathfrak{Q}(S,q)$ satisfies the following properties
\begin{itemize}
	\item [(i)] $\pi_{1}(\mathfrak{Q}(S,q))=S$ where $\pi_{1}$ is the projection in the first coordinate,
	\item[(ii)] for every $b \in S$, $(b,\text{sup}\{y \in [0,1]: (b,y) \in \mathfrak{Q}(S,q)\}) \in \mathfrak{Q}(S,q)$, 
	\item[(iii)] $(b,y) \in \mathfrak{Q}(S,q)$ for every $y \leq \text{sup}\{z \in [0,1]: (b,z) \in \mathfrak{Q}(S,q)\}$, and
	\item[(iv)] if $a \in S$ has a family of decompositions $a=b_{i}c_{i}$ with $i \in J$, then for all $i \in J$,
	$$\text{sup}\{x \in [0,1]: (a,x) \in \mathfrak{Q}(S,q)\} \geq \text{sup}\{y \in [0,1]: (b_{i},y) \in \mathfrak{Q}(S,q)\},$$
	or for all $i \in J$,
	$$\text{sup}\{x \in [0,1]: (a,x) \in \mathfrak{Q}(S,q)\} \geq \text{sup}\{z \in [0,1]: (c_{i},z) \in \mathfrak{Q}(S,q)\}.$$
\end{itemize}
Conversely, every quasi-ideal $Q$ of the direct product semigroup $S \times [0,1]$ which satisfies the following four conditions
\begin{itemize}
	\item [(q-i)] $\pi_{1}(Q)=S$ where $\pi_{1}$ is the projection in the first coordinate,
	\item[(q-ii)] for every $b \in S$, $(b,\text{sup}\{y \in [0,1]: (b,y) \in Q\}) \in Q$,
	\item[(q-iii)] $(b,y) \in Q$ for every $y \leq \text{sup}\{z \in [0,1]: (b,z) \in Q\}$, and
	\item[(q-iv)] If $a \in S$ has a family of decompositions $a=b_{i}c_{i}$ with $i \in J$, then for all $i \in J$,
	$$\text{sup}\{x \in [0,1]: (a,x) \in Q\} \geq \text{sup}\{y \in [0,1]: (b_{i},y) \in Q\},$$
	or for all $i \in J$,
	$$\text{sup}\{x \in [0,1]: (a,x) \in Q\} \geq \text{sup}\{z \in [0,1]: (c_{i},z) \in Q\},$$
\end{itemize}
gives rise to a fuzzy quasi ideal of $S$. Indeed, let $q:S\rightarrow [0,1]$ be such that 
$$q(b)=\text{sup}\{y \in [0,1]: (b,y) \in Q\}.$$
In particular we have that $(b,q(b)) \in Q$. Now we show that $q \circ S \cap S \circ q \subseteq q$. Let $a \in S$ be any fixed element. Denote by $J$ the set of indexes for which there are $b_{i}, c_{i} \in S$ with $i \in J$ such that $a=b_{i}c_{i}$. To prove the above inclusion we need to prove that 
$$\underset{i \in J}{\vee}q(b_{i}) \wedge \underset{j \in J}{\vee}q(c_{j}) \leq q(a),$$
which due to the condition (q-iii) follows directly if we prove that 
$$(a,\underset{(i,j) \in I \times I}{\vee}q(b_{i}) \wedge q(c_{j}))=(a, \underset{i \in I}{\vee}q(b_{i}) \wedge \underset{j \in I}{\vee}q(c_{j})) \in Q.$$
From condition (q-iv) we may assume that for instance $q(a) \geq q(b_{i})$ for all $i \in J$, hence for each $j \in J$, we have from condition (q-iii) that $(a, q(b_{i}) \wedge q(c_{j})) \in Q$. 
Conditions (q-ii) and (q-iii) now imply that
$(a, \underset{(i,j) \in J \times J}{\vee}(q(b_{i}) \wedge q(c_{i}))) \in Q$, and we are done. A similar proof is available if we assume that for all $j \in J$, $q(a) \geq q(c_{j})$. Now we show that $\mathfrak{Q}(S,q)=Q$ for the fuzzy quasi-ideal $q$ defined in terms of the given $Q$. Indeed, 
\begin{align*}
	(b,y) \in \mathfrak{Q}(S,q) & \Leftrightarrow q(b) \geq y \\
	& \Leftrightarrow (b,y) \in Q && \text{(from (q-iii)).}
\end{align*}
In a similar fashion with the case of fuzzy left and right ideals, we denote here by $Q(S \times I)$ the set of all quasi-ideals of $S \times I$ that satisfy conditions (q i-iv), and let $\mathfrak{Q}(S)$ be the set of all fuzzy quasi-ideals of $S$. With these notations we have this
\begin{proposition} 
	The restriction $\Psi|_{\mathfrak{Q}(S)}$ of $\Psi$ on $\mathfrak{Q}(S)$ is a bijection between $\mathfrak{Q}(S)$ and $Q(S \times I)$.
\end{proposition}
\begin{proof}
	The map $\Psi$ sends $q \mapsto \mathfrak{Q}(S,q) \in Q(S \times I)$ and this map is onto from the above. The injectivity of $\Psi|_{\mathfrak{Q}(S)}$ is dealt with in the same way as that of fuzzy one sided ideals.
\end{proof}

\section{Regular Semigroups}

Our next Theorem \ref{osreg} is the fuzzy analogue of theorem 9.3 of \cite{OS} and is partially covert from theorems 3.1.2-3.1.8 of \cite{FS} apart from characterization (iv) which is a new fuzzyfication of its counterpart in theorem 9.3 of \cite{OS}. This theorem characterizes the regularity of a semigroup in terms of fuzzy one sided ideals and fuzzy quasi-ideals but in contrast with \cite{FS} where the proofs are analogue to those of theorem 9.3 of \cite{OS}, our proofs employ the identifications of Proposition \ref{fl=l} and Proposition \ref{fr=r}, to obtain the result as an implication of the corresponding result of theorem 9.3 of \cite{OS}. The following lemma will be useful in the proof of theorem \ref{osreg}.
\begin{lemma} \label{comp}
Let $B,C$ be subsemigroups of a semigroup $S$, and let $\chi_{B}$ and $\chi_{C}$ be the corresponding fuzzy subsemigroups. Then, $\chi_{B} \circ \chi_{C}=\chi_{BC}$.
\end{lemma}
\begin{proof}
For every $a \in S$ we have that,
\begin{align*}
	(\chi_{B} \circ \chi_{C})(a)=1 & \Leftrightarrow\underset{a=st}{\vee}\chi_{B}(s) \wedge \chi_{C}(t)=1 \\
	& \Leftrightarrow \text{ there is a decomposition } a=st, \text{ with } \chi_{B}(s)=1 \text{ and } \chi_{C}(t)=1\\
	& \Leftrightarrow\text{ there is a decomposition } a=st, \text{ with } s \in B \text{ and } t \in C \\
	& \Leftrightarrow a \in BC\\
	& \Leftrightarrow \chi_{BC}(a)=1.
\end{align*} 
Further, if $(\chi_{B} \circ \chi_{C})(a)=0$, then either $a$ does not have a nontrivial decomposition $a=st$, in which case $a \notin BC$ and then $\chi_{BC}(a)=0$, or $a$ decomposes as $a=st$ but for every such decomposition we should have that $\chi_{B}(s) \wedge \chi_{C}(t)=0$. This means that either $s \notin B$ or $t \notin C$, consequently $a \notin BC$ and $\chi_{BC}(a)=0$. Conversely, if $\chi_{BC}(a)=0$, then, either $a$ is indecomposable, or even it is, it can never be written as $a=st$ with both $s \in B$ and $t \in C$, consequently $(\chi_{B} \circ \chi_{C})(a)=0$. Recollecting we have that $\chi_{B} \circ \chi_{C}=\chi_{BC}$.
\end{proof}
\begin{theorem} \label{osreg}
The following are equivalent.
\begin{itemize}
	\item [(i)] $S$ is a regular semigroups.
	\item[(ii)] For every fuzzy right ideal $f$ and every fuzzy left ideal $g$, we have $f \cap g=f \circ g$.
	\item[(iii)] For every fuzzy right ideal $f$ and every fuzzy left ideal $g$ of $S$, (a) $f \circ f=f$, (b) $g \circ g=g$, (c) $f \circ g$ is a fuzzy quasi-ideal of $S$. 
	\item[(iv)] The set $\mathcal{Q}(S)$ of all fuzzy quasi-ideals of $S$ forms a regular semigroup where the multiplication is $\circ$ and for every $q \in \mathcal{Q}(S)$, $q \circ S \circ q=q$..
\end{itemize}
\end{theorem}
\begin{proof}
$(i) \Rightarrow (ii)$: The regularity of $S$ implies that of $S \times I$. Indeed, for every $(a,x) \in S \times I$, let $a' \in S$ such that $aa'a=a$. Then, $(a,x)(a',x)(a,x)=(a,x)$ proving that $(a,x)$ has an inverse. Being regular for $S \times I$ means that for every right ideal $R$ and every left ideal $L$ of $S \times I$, we have that $R \cap L=RL$. We can apply this by taking $R=\mathfrak{R}(S,f)=\Psi(f)$ the fuzzy right ideal of $S$ associated with an arbitrary fuzzy right ideal $f$, and $L=\mathfrak{L}(S,g)=\Psi(g)$ the fuzzy left ideal of $S$ associated with an arbitrary fuzzy left ideal $g$. In this particular case we have $\mathfrak{R}(S,f) \cap \mathfrak{L}(S,g)=\mathfrak{R}(S,f) \mathfrak{L}(S,g)$. Further, for every $a \in S$, we assume that $(f \cap g)(a)=f(a)$ which means that $f(a)=\text{min}(f(a),g(a))$. Under this assumption we will prove that 
$$(f \cap g)(a)=f(a)=\vee_{a=bc}(f(b) \wedge g(c))=(f \circ g)(a).$$
A similar proof can be provided in case $g(a)=\text{min}(f(a),g(a))$. Since $f(a) \leq g(a)$, then $(a,f(a)) \in \mathfrak{R}(S,f) \cap \mathfrak{L}(S,g)$ and $(a,f(a)) \in \mathfrak{R}(S,f) \mathfrak{L}(S,g)$. Hence, there are $(b,x) \in \mathfrak{R}(S,f)$ and $(c,y) \in \mathfrak{L}(S,g)$ such that $(a,f(a))=(b,x)(c,y)$. It follows that $a=bc$, and that
$$f(a)=x \wedge y \leq f(b) \wedge g(c).$$
But $f$ is a fuzzy right ideal, then $f(b) \leq f(bc)=f(a)$, hence
$$f(b) \leq f(a) \leq f(b) \wedge g(c) \leq f(b),$$
which implies that $f(a)=f(b)=f(b) \wedge g(c)$. Now we write $(f \circ g)(a)$ as
$$(f \circ g)(a)=(f(b) \wedge g(c)) \vee (\vee_{a=b'c', b' \neq b}(f(b') \wedge g(c'))),$$
and obtain from the above that 
$$(f \circ g)(a)=f(a) \vee (\vee_{a=b'c', b' \neq b}(f(b') \wedge g(c'))) \geq f(a).$$
But on the other hand we have that
\begin{align*}
	(f \circ g)(a)= \vee_{a=b''c''}(f(b'') \wedge g(c'')) & \leq \vee_{a=b''c''} f(b'')\\
	& \leq f(b''c'') && \text{(since $f$ is a fuzzy right ideal)}\\
	&= f(a).
\end{align*}
Therefore we finally have that $(f\circ g)(a)=f(a)=(f \cap g)(a)$.\\
$(ii) \Rightarrow (i)$: Let $R$ and $L$ be arbitrary right and left ideals of $S$ and $\chi_{R}$ and $\chi_{L}$ be their respective fuzzy right and left ideals. From the assumption we have that $\chi_{R} \cap \chi_{L}= \chi_{R} \circ \chi_{L}$. On the one hand, it is obvious that $\chi_{R} \cap \chi_{L}=\chi_{R \cap L}$, and on the other hand we have from lemma \ref{comp} that $\chi_{R} \circ \chi_{L}= \chi_{RL}$. Combining both equalities we derive that $R \cap L= RL$. Now theorem 9.3 \cite{OS} implies (i).\\
$(i) \Rightarrow (iii)$: Let $f$ be a fuzzy right ideal of $S$ and let $\mathfrak{R}(f,S)=\Psi(f)$ be the respective right ideal of $S \times I$. Since $S$ is regular, so is $S \times I$ and then (iii) of theorem 9.3 \cite{OS} implies that $\mathfrak{R}(f,S) \mathfrak{R}(f,S)=\mathfrak{R}(f,S)$. We imply from this that $f \circ f=f$. For every $a \in S$, $(a,f(a)) \in \mathfrak{R}(f,S)$, hence there are $(b,y), (c,z) \in \mathfrak{R}(f,S)$ such that 
$$(a,f(a))=(b,y) (c,z)=(bc, y \wedge z).$$
In particular we have that $a$ has a decomposition $a=bc$ and that $$f(a) = y \wedge z \leq f(b) \wedge f(c).$$
It follows from this that 
$$(f \circ f)(a) = (f(b) \wedge f(c)) \vee \left(\underset{\underset{(b',c') \neq (b,c)}{a=b'c'}}\vee f(b') \wedge f(c')\right) \geq f(a).$$
On the other hand we have that
\begin{align*}
(f \circ f)(a) &= \underset{a=b'c'}{\vee} f(b') \wedge f(c')\\
& \leq \underset{a=b'c'}{\vee} f(a) \wedge f(c') && \text{ (since $f$ is a fuzzy right ideal) } \\
& \leq f(a).
\end{align*}
Combining both inequalities we obtain that $(f \circ f)(a)=f(a)$.\\
In a similar fashion one can prove that for every fuzzy left ideal $g$ we have that $g \circ g=g$. Finally, that $f \circ g$ is a fuzzy quasi-ideal of $S$ for every fuzzy right ideal $f$ and every fuzzy left ideal $g$, follows from Lemma 2.6.5 \cite{FS} and from the equality $f \circ g=f \cap g$ which is a consequence of the equivalence $(i) \Leftrightarrow (ii)$.\\
$(iii) \Rightarrow (i)$: Under the assumptions that for every fuzzy right ideal $f$ and every fuzzy left ideal $g$, we have that $f \circ f=f$, $g \circ g=g$ and $f \circ g$ is a fuzzy quasi-ideal, we prove that for every right ideal $R$ of $S$ and every left ideal $L$ of $S$ we have that $RR=R$, $LL=L$ and that $RL$ is a quasi-ideal of $S$. These three conditions imply from theorem 9.3 of \cite{OS} that $S$ is regular. Consider now the fuzzy right ideal $\chi_{R}$ and the fuzzy left ideal $\chi_{L}$ for which we have that $\chi_{R} \circ \chi_{R}=\chi_{R}$, $\chi_{L} \circ \chi_{L}=\chi_{L}$ and that $\chi_{R} \circ \chi_{L}$ is a fuzzy quasi-ideal of $S$. Lemma \ref{comp} and the first two assumptions imply immediately that $RR=R$ and $LL=L$. Again lemma \ref{comp} implies that $\chi_{R} \circ \chi_{L}=\chi_{RL}$, which from our assumption is a fuzzy quasi-ideal of $S$. Then lemma 2.6.4 of \cite{FS} implies $RL$ is a quasi-ideal of $S$.\\
$(i) \Rightarrow (iv)$: Let $q_{1},q_{2} \in \mathcal{Q}(S)$ be arbitrary and $\mathfrak{Q}(q_{1},S)$, $\mathfrak{Q}(q_{2},S)$ be their corresponding quasi ideals of $S \times I$. We know from theorem 9.3 \cite{OS} that $\mathfrak{Q}(q_{1},S)\mathfrak{Q}(q_{2},S)$ is a quasi-ideal of $S \times I$ since $S \times I$ is regular. Let $a \in S$ be arbitrary and want to show that
$$((q_{1} \circ q_{2}) \circ S \wedge S \circ (q_{1} \circ q_{2})) (a) \leq (q_{1} \circ q_{2})(a).$$
Assume now that $a$ decomposes as $a=b_{i}c_{i}d_{i}$ where $i \in J$, so the above inequality now writes as
\begin{equation*} 
\underset{i \in J}{\vee}( q_{1}(b_{i}) \wedge q_{2}(c_{i})) \wedge \underset{i \in J}{\vee}( q_{1}(c_{i}) \wedge q_{2}(d_{i})) \leq (q_{1} \circ q_{2})(a),
\end{equation*} 
which is equivalent to
\begin{equation*} 
	\underset{(i,j) \in J\times J}{\vee}( q_{1}(b_{i}) \wedge q_{2}(c_{i}) \wedge  q_{1}(c_{j}) \wedge q_{2}(d_{j})) \leq (q_{1} \circ q_{2})(a).
\end{equation*} 
This follows if we prove that for every $(i,j) \in J \times J$ we have
\begin{equation} \label{2bp}
 q_{1}(b_{i}) \wedge q_{2}(c_{i}) \wedge  q_{1}(c_{j}) \wedge q_{2}(d_{j}) \leq (q_{1} \circ q_{2})(a).
\end{equation}
For this consider the element $(a, q_{1}(b_{i}) \wedge q_{2}(c_{i}) \wedge  q_{1}(c_{j}) \wedge q_{2}(d_{j})) \in S \times I$ for which we have
\begin{align*}
\mathfrak{Q}(q_{1},S) \mathfrak{Q}(q_{2},S) (S \times I) &\ni (b_{i},q_{1}(b_{i}))(c_{i},q_{2}(c_{i}))(d_{i}, q_{1}(c_{j}) \wedge q_{2}(d_{j}))\\
&=(a, q_{1}(b_{i}) \wedge q_{2}(c_{i}) \wedge  q_{1}(c_{j}) \wedge q_{2}(d_{j}))\\
&= (b_{j}, q_{1}(b_{i}) \wedge q_{2}(c_{i})) (c_{j}, q_{1}(c_{j})) (d_{j},q_{2}(d_{j}))\\
&\in (S \times I) \mathfrak{Q}(q_{1},S) \mathfrak{Q}(q_{2},S). 
\end{align*}
Since $\mathfrak{Q}(q_{1},S) \mathfrak{Q}(q_{2},S) (S \times I) \cap (S \times I) \mathfrak{Q}(q_{1},S) \mathfrak{Q}(q_{2},S) \subseteq \mathfrak{Q}(q_{1},S) \mathfrak{Q}(q_{2},S)$, then
$$(a, q_{1}(b_{i}) \wedge q_{2}(c_{i}) \wedge  q_{1}(c_{j}) \wedge q_{2}(d_{j})) \in \mathfrak{Q}(q_{1},S) \mathfrak{Q}(q_{2},S),$$
therefore there are $(s,x) \in \mathfrak{Q}(q_{1},S)$ and $(t,y) \in \mathfrak{Q}(q_{2},S)$ such that
\begin{equation} \label{key}
(a, q_{1}(b_{i}) \wedge q_{2}(c_{i}) \wedge  q_{1}(c_{j}) \wedge q_{2}(d_{j}))=(s,x)(t,y)=(st,x \wedge y).
\end{equation}
It follows that
\begin{align*}
(q_{1}\circ q_{2})(a) &= (q_{1}\circ q_{2})(st) && \text{(since $a=st$)}\\
&\geq q_{1}(s) \wedge q_{2}(t) && \text{(from the definition of $\circ$)} \\
&\geq x \wedge y && \text{(since $q_{1}(s) \geq x$ and $q_{2}(t) \geq y$)}\\
&= q_{1}(b_{i}) \wedge q_{2}(c_{i}) \wedge  q_{1}(c_{j}) \wedge q_{2}(d_{j}) && \text{(from (\ref{key})).}
\end{align*}
This proves (\ref{2bp}) and we are done with the first part of the proof. It remains to prove that $(\mathcal{Q}(S), \circ)$ is regular. For this we utilize again theorem 9.3 \cite{OS} where it is proved that for every quasi-ideal $Q$ of $S$ we have that $QSQ=Q$. Let $q \in \mathcal{Q}(S)$ and $\mathfrak{Q}(q,S)$ its corresponding quasi-ideal of $S \times I$ which satisfies
$$\mathfrak{Q}(q,S) (S \times I) \mathfrak{Q}(q,S)=\mathfrak{Q}(q,S).$$
We use this and the obvious fact that $S \in \mathcal{Q}(S)$ to show that $q \circ S \circ q=q$ which would prove that $q$ is regular. First we show that for every $a \in S$, $(q \circ S \circ q)(a) \leq q(a)$. Indeed, since
$$(q \circ S \circ q)(a)=\underset{a=bcd}{\vee}q(b) \wedge q(d),$$
to prove the inequality,  it is enough to show that $q(a) \geq q(b)$ or $q(a) \geq q(d)$ for every decomposition $a=bcd$. But either one or the other is insured from lemma \ref{lm} and then the result follows. Conversely, $(a,q(a)) \in \mathfrak{Q}(q,S)$, therefore there are $(b,x),(d,y) \in \mathfrak{Q}(q,S)$ and $(c,z) \in S \times I$ such that
$$(a,q(a))=(b,x)(c,z)(d,y)=(bcd,x \wedge z \wedge y).$$ 
From this it follows that 
\begin{align*}
q(a)& = x \wedge z \wedge y \\
& \leq q(b) \wedge q(d) && \text{(since $x \leq q(b)$ and $y \leq q(d)$)}\\
& \leq (q \circ S \circ q)(a) && \text{(since $a=bcd$,)}
\end{align*}
hence $q(a) \leq (q \circ S \circ q)(a)$ which concludes the proof. \\
$(iv) \Rightarrow (i)$: If we prove that for every quasi-ideal $Q$ of $S$ we have that $QSQ=Q$, then Theorem 9.3 of \cite{OS} implies that $S$ is regular. Consider $\chi_{Q}$ which from Lemma 2.6.4 of \cite{FS} is a fuzzy quasi-ideal of $S$ for which we have that 
$\chi_{Q}=\chi_{Q} \circ \chi_{S} \circ \chi_{Q}$. From Lemma \ref{comp} we can write the above as $\chi_{Q}=\chi_{QSQ}$ which implies that $Q=QSQ$ and we are done.
\end{proof}

\section{Semilattice decompositions}

Recall from \cite{Saito} that a semigroup $S$ is called left regular if for every $a \in S$, there is $x \in S$ such that $a=xa^{2}$. Likewise regular semigroups, left regular ones have the property that every element have nontrivial decompositions as a product of two elements. Most importantly, as shown in \cite{Saito}, the left regular semigroups whose every left ideal is two-sided, decompose as semilattices of left simple semigroups. We will generalize this by introducing here fuzzy semilattices of left simple fuzzy subsemigroups. Before we give our next definition, we recall that for every fuzzy subsemigroup $f$ of a given semigroup $S$ and every $t \in I$, there is the so called level subset (or the $t$-cut) $f_{t}=\{a \in S: f(a) \geq t\}$ associated with $t$. It is obvious that $f_{t}$ is a subsemigroup of $S$.
\begin{definition} \label{fsl}
	Let $(Y,\cdot)$ be a semilattice and $f_{\alpha}$ for $\alpha \in Y$ be a family of fuzzy subsemigroups of a semigroup $S$. We say that this family forms a fuzzy semilattice of fuzzy subsemigroups of $S$ if the following three conditions are satisfied.\begin{itemize}
		\item [(i)] $\forall a \in S$, and all $\alpha, \beta \in Y$ with $\alpha \neq \beta$, we have that $f_{\alpha}(a) \wedge f_{\beta}(a)=0$.
		\item[(ii)] For all $\alpha, \beta \in Y$, $f_{\alpha} \circ f_{\beta} \subseteq f_{\alpha \cdot \beta}$.
		\item[(iii)] $\forall (\alpha, t)  \in Y \times I^{\ast}$, $(f_{\alpha})_{t} \neq \emptyset$, and $\forall(a,t) \in S \times I^{\ast}$, $\exists \alpha \in Y$ such that $a \in (f_{\alpha})_{t}$.
	\end{itemize}
In such a case we say that $S$ is a fuzzy semilattice of fuzzy subsemigroups of $S$. We observe the following.
\begin{proposition} \label{sd-fsd}
	If a semigroup $S$ has a semilattice decomposition, then $S$ has a fuzzy semilattice decomposition.
\end{proposition}
\begin{proof}
	Assume that $S$ is a semilattice of semigroups $S_{\alpha}$ with $\alpha \in Y$ where $Y$ is a semilattice. Let $\chi_{S_{\alpha}}$ be the corresponding fuzzy subsemigroups of $S$. We show that the family $\chi_{S_{\alpha}}$ with $\alpha \in Y$ forms a fuzzy semilattice of left simple fuzzy subsemigroups of $S$. Let $\alpha \neq \beta$ be elements of $Y$ and $a \in S$ be arbitrary. If $a \notin S_{\alpha} \cup S_{\beta}$, then $\chi_{S_{\alpha}}(a) \wedge \chi_{S_{\beta}}(a)=0$. If $a \in S_{\alpha}$, then $\chi_{S_{\alpha}}(a) \wedge \chi_{S_{\beta}}(a)=0$, and similarly, if $a \in S_{\beta}$, then $\chi_{S_{\alpha}}(a) \wedge \chi_{S_{\beta}}(a)=0$. Secondly, for all $\alpha, \beta \in Y$, we have that 
	\begin{align*}
		\chi_{S_{\alpha}} \circ \chi_{S_{\beta}}&=\chi_{S_{\alpha}S_{\beta}} && \text{(from Lemma \ref{comp})} \\
		&\subseteq \chi_{S_{\alpha \cdot \beta}} && \text{(since $S_{\alpha}S_{\beta} \subseteq S_{\alpha \cdot \beta}$).}
	\end{align*}
	Thirdly, $\forall (\alpha, t)  \in Y \times I^{\ast}$, $(\chi_{S_{\alpha}})_{t} \neq \emptyset$ since for $t\neq 0$, $(\chi_{S_{\alpha}})_{t} =S_{\alpha}$. Also $\forall(a,t) \in S \times I^{\ast}$, $\exists \alpha \in Y$ such that $a \in (\chi_{S_{\alpha}})_{t}$. We can chose $\alpha \in Y$ such that $a \in S_{\alpha}$, and then $a \in S_{\alpha}= (\chi_{S_{\alpha}})_{t}$.
\end{proof}
\end{definition}
\begin{proposition} \label{f2s}
	If $S$ is a fuzzy semilattice of fuzzy subsemigroups of $S$, then $S \times I^{\ast}$ has a semilattice decomposition.
\end{proposition}
\begin{proof}
	Assume that $S$ is a fuzzy semilattice of fuzzy subsemigroups $f_{\alpha}$ with $\alpha \in Y$. We define for each $\alpha \in Y$ the set
	$$\mathfrak{S}^{\ast}(f_{\alpha},S)=\{(a,t) \in S \times I^{\ast}: f_{\alpha}(a) \geq t\}.$$
	This set is nonempty from the assumption (iii) on $f_{\alpha}$, and clearly forms a subsemigroup of $S \times I^{\ast}$. For every fixed pair $(\alpha,t) \in Y \times I^{\ast}$ we consider the nonempty set
	$$(f_{\alpha})_{t}=\{a \in S: f_{\alpha}(a) \geq t\}=\{a \in S: (a,t) \in \mathfrak{S}^{\ast}(f_{\alpha},S)\},$$
and let
	$$f_{(\alpha,t)}= (f_{\alpha})_{t} \times \{t\}.$$
	It is straightforward that $f_{(\alpha,t)}$ is a subsemigroup of $S \times I^{\ast}$. We will prove that the family of all $f_{(\alpha,t)}$ forms a semilattice decomposition of subsemigroups for $S \times I^{\ast}$, where the index semilattice is $(Y,\cdot) \times (I^{\ast}, \wedge)$. We observe that for each $\alpha \in Y$, 
	\begin{equation} \label{prelu}
		\mathfrak{S}^{\ast}(f_{\alpha},S)=\underset{t \in I^{\ast}}{\cup} f_{(\alpha,t)}.
	\end{equation}
	With this observation in mind we prove that the family of all $f_{(\alpha,t)}$ with $(\alpha,t) \in Y \times I^{\ast}$, are pairwise disjoint. Indeed, for every fixed $\alpha \in Y$, two different $f_{(\alpha,t)},f_{(\alpha,t')}$ do not intersect since $t \neq t'$. To prove that for all $\alpha \neq \beta$ and $t,t' \in I^{\ast}$, $f_{(\alpha,t)} \cap f_{(\beta,t')}=\emptyset$ it is enough to prove that $\mathfrak{S}^{\ast}(f_{\alpha},S) \cap \mathfrak{S}^{\ast}(f_{\beta},S) =\emptyset$. Let for this $(a,t) \in \mathfrak{S}^{\ast}(f_{\alpha},S) \cap \mathfrak{S}^{\ast}(f_{\beta},S)$, then 
	$$f_{\alpha}(a) \geq t \text{ and } f_{\beta}(a) \geq t,$$
	which is impossible since one of $f_{\alpha}(a)$ or $f_{\beta}(a)$ has to be zero from definition \ref{fsl}, (i). So it remains that $\mathfrak{S}^{\ast}(f_{\alpha},S) \cap \mathfrak{S}^{\ast}(f_{\beta},S) =\emptyset$. Let us prove that for every $(\alpha,t), (\beta,t') \in Y \times I^{\ast}$ we have that $f_{(\alpha,t)} f_{(\beta,t')} \subseteq f_{(\alpha \cdot \beta,t \wedge t')}$. Indeed, for every $(a,t) \in f_{(\alpha,t)}$ and $(b,t') \in f_{(\beta,t')}$ we have that
	\begin{align*}
		f_{\alpha \cdot \beta}(ab)& \geq (f_{\alpha} \circ f_{\beta})(ab) && \text{(from definition \ref{fsl}, (ii))}\\
		& = \underset{a'b'=ab}{\vee}f_{\alpha}(a')\wedge f_{\beta}(b')\\
		& \geq f_{\alpha}(a)\wedge f_{\beta}(b)\\
		& \geq t \wedge t',
	\end{align*}
	which implies that 
	$$(a,t)(b,t')=(ab,t \wedge t') \in f_{(\alpha \cdot \beta,t \wedge t')}.$$
	Finally, that $\underset{(\alpha, t) \in Y \times I^{\ast}}{\cup}f_{(\alpha,t)}=S \times I^{\ast}$ follows directly from definition \ref{fsl}, (iii). 
\end{proof}
\begin{definition} \label{flsimp}
	A fuzzy subsemigroup $f$ of a semigroup $S$ is called a left simple fuzzy subsemigroup of $S$ if for every $t \in I^{\ast}$, the set $f_{t}=\{a \in S: f(a) \geq t\}$ is nonempty and forms a left simple subsemigroup of $S$.
\end{definition}
This definition is consistent with the definition of a left simple subsemigroup $A$ of $S$. Indeed, in that case the property of the fuzzy subsemigroup $\chi_{A}$ that $(\chi_{A})_{t}=\{a \in S: \chi_{A}(a) \geq t\}$ is left simple for all $t \in I^{\ast}$ means exactly that $A$ is a left simple subsemigroup of $S$ since for $t \in I^{\ast}$, $(\chi_{A})_{t}=A$.
\begin{theorem} \label{tsaito+fuzzy}
	For a semigroup $(S, \cdot)$ the following conditions
	are equivalent:
	\begin{itemize}
		\item[(1')] $S$ is a fuzzy semilattice of left simple fuzzy subsemigroups of $S$.
		\item[(1)] $S$ is a semilattice of left simple semigroups.
	\end{itemize}
\end{theorem}
\begin{proof}
	We begin by proving that $(1') \Rightarrow (1)$. Assumption (1') and Proposition \ref{f2s} imply that there is a semilattice decomposition of $S \times I^{\ast}$ with components $f_{(\alpha, t)}=(f_{\alpha})_{t} \times \{t\}$ where $(\alpha,t) \in Y \times I^{\ast}$ and $Y$ is a semilattice.  Further we show that each $f_{(\alpha,t)}$ is a left simple subsemigroup of $S \times I^{\ast}$. Indeed, since $f_{(\alpha,t)}=(f_{\alpha})_{t} \times \{t\}$ and 
	$$(f_{\alpha})_{t}=\{a \in S: (a,t) \in \mathfrak{S}^{\ast}(f_{\alpha},S)\}=\{a \in S: f_{\alpha}(a) \geq t\},$$
	we have from our assumption and definition \ref{flsimp} that $f_{(\alpha,t)}$ is left simple. Saito's theorem now implies that $S \times I^{\ast}$ is left regular and that every left ideal there is two sided. The first implies that $S$ is also left regular, and the second implies that every left ideal of $S$ is two sided. To see the second one, we let $L$ be a left ideal of $S$, and $t \in I^{\ast}$, then clearly $L \times (0,t]$ is a left ideal of $S\times I^{\ast}$ and hence a right ideal. Let now $s \in S$ and $a \in L$ be arbitrary. Then for every $(a,t') \in L \times (0,t]$ and $(s,u) \in S \times I^{\ast}$ we have that $(as,t' \wedge u)=(a,t')(s,u) \in L \times (0,t]$, consequently $as \in L$ which proves that $L$ is a right ideal of $S$.

	Let us now prove the converse implication $(1) \Rightarrow (1')$. Assume that $S$ is a semilattice of left simple semigroups $S_{\alpha}$ with $\alpha \in Y$ where $Y$ is a semilattice. From Proposition \ref{sd-fsd}, we have that the family $\chi_{S_{\alpha}}$ with $\alpha \in Y$ forms a fuzzy semilattice of fuzzy subsemigroups of $S$. In addition, each $\chi_{S_{\alpha}}$ is a left simple fuzzy subsemigroup. This is obvious from the comment after definition \ref{flsimp}. 
\end{proof} \newline

Theorem 4.1.3 of \cite{JMH} states that \textit{Every completely regular semigroup is a semilattice of completely simple semigroups}. Our next attempt will be to generalize this by replacing the semilattice of completely simple semigroups with fuzzy semilattice of completely simple fuzzy subsemigroups. We make first the following definition.
\begin{definition} \label{fcs}
	A fuzzy subsemigroup $f$ of a semigroup $S$ is called a completely simple fuzzy subsemigroup of $S$ if for each $t \in I^{\ast}$, the set $f_{t}=\{a \in S: f(a) \geq t\}$ is a nonempty and forms a completely simple subsemigroup of $S$.
\end{definition}
As in the case of left simple fuzzy subsemigroups, the above definition generalizes completely simple subsemigroups of $S$, for if $A$ is such one, then $\chi_{A}$ has the property that for each $t \in I^{\ast}$, $(\chi_{A})_{t}=\{a \in S: \chi_{A}(a) \geq t\}=A$ is a completely simple subsemigroup.

\begin{theorem}
	Every completely regular semigroup $S$ is a fuzzy semilattice of completely simple fuzzy subsemigroups.
\end{theorem}
\begin{proof}
	If we prove that $S$ being a fuzzy semilattice of completely simple fuzzy subsemigroups is equivalent to $S$ being a semilattice of completely simple subsemigroups, then our claim follows from Theorem 4.1.3 of \cite{JMH}. Assume that there is a semilattice $Y$ and a family $f_{\alpha}$ with $\alpha \in Y$, of completely simple fuzzy subsemigroups of $S$ which form a fuzzy semilattice decomposition for $S$. We first show that $S \times I^{\ast}$ has a semilattice decomposition of completely simple subsemigroups. Indeed, from Proposition \ref{f2s} we have that the family $f_{(\alpha,t)}=(f_{\alpha})_{t} \times \{t\}$ with $(\alpha, t) \in Y \times I^{\ast}$ is a semilattice decomposition of $S \times I^{\ast}$. Our assumption together with definition \ref{fcs} imply that each $f_{(\alpha,t)}$ is a completely simple subsemigroup of $S \times I^{\ast}$. Secondly, Proposition 4.1.2 of \cite{JMH} implies that each $f_{(\alpha,t)}$ is completely regular.
	Thirdly, let $a \in S$ be arbitrary. For every $t \in I^{\ast}$, from our assumption of the family $f_{\alpha}$, there is $\alpha \in Y$ such that $(a,t) \in f_{(\alpha,t)}$. Since $f_{(\alpha,t)}$ is completely regular, it follows that $(f_{\alpha})_{t}$ is completely regular as well, therefore there is a subgroup of $(f_{\alpha})_{t}$ containing $a$. This proves that $S$ is completely regular. Theorem 4.1.3 of \cite{JMH} implies that $S$ is a semilattice of completely simple subsemigroups. For the converse, assume that $S$ is a semilattice $Y$ of completely simple subsemigroups $S_{\alpha}$ with $\alpha \in Y$. Consider the family of fuzzy subsemigroups $\chi_{S_{\alpha}}$ with $\alpha \in Y$. From Proposition \ref{sd-fsd} we have that this family forms a fuzzy semilattice of fuzzy subsemigroups of $S$. The comment after definition \ref{fcs} implies that each $\chi_{S_{\alpha}}$ is fuzzy completely simple.
\end{proof}

\end{document}